\documentclass[a4paper,10pt]{article}
\usepackage[plainpages=false]{hyperref}
\usepackage{amsfonts,latexsym,rawfonts,amsmath,amssymb,amsthm, mathrsfs, lscape}
\usepackage{verbatim}

\usepackage[all]{xy}
\usepackage{graphicx,psfrag}

\usepackage{array, tabularx}

\usepackage{setspace}

\newtheorem{thm}{Theorem}

\newtheorem{lem}{Lemma}[section]

\newtheorem{prop}{Proposition}[section]
\newtheorem{prob}[thm]{Question}

\theoremstyle{remark}

\theoremstyle{definition}
\newtheorem{defi}{Definition}[section]

\numberwithin{equation}{section}

\def\p{\partial}
\def\R{\mathbb{R}}
\def\C{\mathbb{C}}

\def\Z{\mathbb{Z}}

\def\l{\lambda}

\def\i{\sqrt{-1}}

\def\o{\omega}

\def\g{{\mathfrak g}}

\def\cD{\mathcal D}

\def\cF{{\mathcal F}}

\def \Ker {\text{Ker}}

\begin{document}

\title{The generalized Frankel conjecture in Sasaki geometry}

\author{Weiyong He and Song Sun}

\maketitle

\begin{abstract}
We prove some structure results for \emph{transverse reducible} Sasaki manifolds. In particular, we show Sasaki manifolds with positive Ricci curvature is transversely irreducible, and so there is no join (product) construction for irregular Sasaki-Einstein manifolds, as opposed to the quasi-regular case done by Wang-Ziller and Boyer-Galicki.  As an application, we classify compact Sasaki manifolds with non-negative transverse bisectional curvature, which can be viewed as the generalized Frankel conjecture (N. Mok's theorem) in Sasaki geometry.
\end{abstract}

\section{Introduction} The de Rham decomposition theorem asserts that a simply connected complete Riemannian manifold with reducible holonomy group must split as the product of two Riemannian manifolds; hence the building blocks of Riemannian manifolds are irreducible ones.
 In K\"ahler geometry, this decomposition theorem is compatible with the K\"ahler structure, so the building blocks are then irreducible K\"ahler manifolds. While the procedure of taking products of two K\"ahler manifolds is quite straightforward, it becomes much more interesting when one allows orbifold singularities and look at certain $S^1$ bundle as a Sasaki manifold. By \cite{WZ}, \cite{BG1} this construction (called ``join" construction in \cite{BG1}) shows the diversity of quasi-regular Sasaki-Einstein manifolds.  Motivated by the Ads/CFT correspondence in theoretical physics, the first example of an irregular Sasaki-Einstein manifold was constructed in \cite{GMSW}.  However, we will show that join construction cannot be applied to irregular Sasaki-Einstein manifolds. This phenomenon  follows from a general structure theorem for irregular Sasaki structures. Roughly speaking, we show (see Theorem \ref{thm-4-1}) that a compact irregular Sasaki manifold with certain curvature assumption is irreducible  (for the definition of irreducibility in Sasaki setting, we refer to Section \ref{S2}).

As an application, we can extend the classification results in \cite{He-Sun} to the borderline case. In \cite{He-Sun} we proved  that a compact simply connected Sasaki manifold with positive transverse bisectional curvature is a \emph{simple deformation} of the standard Sasaki structure on $S^{2n+1}$. Allowing the transverse bisectional curvature to be non-negative, we obtain

\begin{thm}\label{T-general}
Let $(M, \xi, \eta, g)$ be a compact Sasaki manifold with non-negative transverse bisectional curvature of dimension $2n+1$, then one of the following is true,

\begin{enumerate}
\item $M$ is  irregular. Then $\pi_1(M)$ is  finite and the universal cover of $M$ is isomorphic to a weighted Sasaki sphere $(S^{2n+1}, \xi_0, \eta_0, g_0)$ with nonnegative transverse bisectional curvature.

\item $M$ is quasi-regular. Then the universal cover of the quotient orbifold $M/\cF_\xi$ is isomorphic to
\[(WP_1, \omega_1)\times \cdots \times (WP_k, \omega_k)\times (O_1, h_1)\times \cdots \times (O_l, h_l)\times (\mathbb{C}^i, h_0)\]
where $\omega_j$ is a K\"ahler metric on the weighted projective space $WP_j$ with nonnegative bisectional curvature,  $O_1$, $\cdots$, $O_l$ are compact irreducible Hermitian symmetric spaces of rank $\geq 2$ endowed with the canonical metric and $h_0$ is the flat metric on $\C^i$.
\end{enumerate}
\end{thm}

There are several technical ingredients in the proof the above theorem. 
When $M$ is quasi-regular, this is  actually an orbifold version of  classification of compact K\"ahler  manifolds with nonnegative bisectional curvature, known as the generalized Frankel conjecture. The generalized Frankel conjecture can be reduced to the special case that the manifold is assumed to be simply connected and with second Betti number one, using  the structure theorem of Howard-Smyth-Wu \cite{HSW} for compact K\"ahler manifolds with nonnegative bisectional curvature; and it  was proved by Bando \cite{Bando} in complex dimension three, and later by  Mok \cite{Mok} in general. While if $M$ is assumed to be quasi-regular, 
 similar structure theorem as in \cite{HSW} for Sasaki setting can also be proved with appropriate modifications.  

Recently Brendle-Schoen \cite{BS1} proved the diffeomorphism sphere theorem when the sectional curvature of a Riemannian manifold $(M, g)$ is $1/4$ pinched; they also classified the weakly $1/4$ pinched manifolds \cite{BS2}. Their strategy to deal with the later case  has been adapted recently by H-L. Gu \cite{Gu} to  provide an alternative proof  of Mok's theorem. 
An appropriate modification of Gu's argument \cite{Gu} together with our previous results in \cite{He-Sun} implies the following special case of Theorem \ref{T-general}. 

\begin{thm}\label{T-special}Let $(M, \xi, \eta, g)$ be a compact simply connected Sasaki manifold with non-negative transverse bisectional curvature such that $b_2^B=b_2(M)+1=1$, then  either $M$ is a weighted Sasaki sphere with a simple metric, or $M$ is regular, and the quotient manifold is isometrically biholomorphic to  an compact irreducible Hermitian symmetric spaces of rank $\geq 2$ endowed with the canonical metric.
\end{thm}

The  new ingredient for Theorem \ref{T-general} is  then reduced to the structure theorem for Sasaki manifolds discussed above; namely,  an irregular Sasaki manifold satisfying the curvature assumption in Theorem \ref{T-general} must be  irreducible. 
Such a property seems to be topological in nature; in general we would like to ask that 
 
 \begin{prob}
Is every (simply-connected) compact irregular Sasaki  manifold  irreducible? 
\end{prob}

This article is organized as follows. In Section 2 we discuss some general theory in Sasaki setting and prove Theorem \ref{thm-4-1}. In Section 3 we prove Theorem \ref{T-special} and Theorem \ref{T-general}. 

\section{Splitting phenomenon in Sasaki geometry}\label{S2}

We first briefly recall notions in Sasaki geometry.
We always denote a Sasaki structure on $M$ by $(\xi, \eta, g)$,  with $\xi$ the Reeb vector field, $\eta$ the contact one form and $g$ the Sasaki metric. There is an orthogonal decomposition of the tangent bundle \[TM=L\xi\oplus \cD,\] where $L\xi$ is the trivial  bundle generalized by $\xi$, and $\cD=\Ker (\eta)$.
The metric $g$ and the contact form $\eta$ determine a $(1,1)$ tensor field $\Phi$ on $M$ by
\[
g(Y, Z)=\frac{1}{2} d\eta(Y, \Phi Z), Y, Z\in \Gamma(\cD).
\]
$\Phi$ restricts to an almost complex structure on $\cD$: \[\Phi^2=-\mathbb{I}+\eta\otimes \xi. \]
Since both $g$ and $\eta$ are invariant under $\xi$, there is a well-defined K\"ahler structure on the local leaf space of the Reeb foliation. This is called a \emph{transverse K\"ahler structure} $g^T$ and it is induced by $(\cD , \omega^T, \Phi|_{\cD})$, where the transverse K\"ahler form is given by
\[
\o^T=\frac{1}{2}d\eta.
\]
Hence the Sasaki structure is related to the transverse structure as follows,
\[
g=\eta\otimes \eta+g^T.
\]
Denote $\nu=\nu(\cF_\xi)$ to be the quotient bundle of the Reeb foliation generated by $\xi$, namely, $\nu(\cF_\xi)=TM/L_\xi$. The transverse K\"ahler metric $g^T$ induces a natural bundle isomorphism $f: \nu(\cF_{\xi})\rightarrow \cD$ which splits the exact sequence
\[
0\rightarrow L_\xi\rightarrow TM\rightarrow \nu(\cF_\xi)\rightarrow 0.
\] 
 The tensor field $\Phi$ naturally induces a splitting on $\cD\otimes \C=\cD^{1, 0}\oplus \cD^{0, 1}$; similarly we have $\nu\otimes \C=\nu^{1, 0}\oplus \nu^{0, 1}$. The map $f$  induces a bundle isomorphism $\nu^{1, 0}\rightarrow \cD^{1, 0}$.

We also recall transverse K\"ahler structure in local coordinates; see \cite{FOW} Section 3 for a nice reference.
Let $(M, \xi, \eta, g)$ be a Sasaki manifold.
Let $U_{\alpha}$ be an open covering of $M$, $\pi_{\alpha}: U_{\alpha}\rightarrow V_{\alpha}\subset \C^n$ the submersion corresponding to the Reeb foliation such that $\pi_{\alpha}\circ \pi_{\beta}^{-1}: \pi_{\beta}(U_{\alpha}\cap U_{\beta})\rightarrow \pi_{\alpha}(U_{\alpha}\cap U_{\beta})$ is bi-holomorphic when $U_{\alpha}\cap U_{\beta}$ is non-empty.  Let  $( z_1, \cdots, z_n)$ be a holomorphic coordinate in $V_{\alpha}$, and $(x, z_1, \cdots, z_n)$ the corresponding coordinate on $U_{\alpha}$ such that $\p_x=\xi$. The transverse K\"ahler metric $g^T$ induces a K\"ahler metric $g_{\alpha}^T$ on $V_{\alpha}$;
on $U_\alpha \cap U_\beta$,
\begin{equation}\label{E-cocycle}
\pi_\alpha\circ \pi^{-1}_\beta: \pi_\beta(U_\alpha\cap U_\beta)\rightarrow \pi_\alpha (U_\alpha\cap U_\beta)
\end{equation}
gives an isometry of K\"ahler manifolds $(V_\alpha, g^T_\alpha)$ and $(V_\beta, g^T_\beta)$. In particular, the collection $\{V_\alpha, g^T_\alpha\}$ does not give rise to a K\"ahler manifold, but it still satisfies the cocycle condition. 

\begin{defi}A Sasaki manifold $(M, \xi, \eta, g)$ is \emph{locally transverse reducible} at a point $p\in M$, if there exists a neighborhood $U_\alpha$ of $p$ and the submersion $\pi_\alpha: U_\alpha\rightarrow V_\alpha$, such that the transverse K\"ahler metric $\{V_\alpha, g^T_\alpha \}$ is reducible. We say $(M, \xi, \eta, g)$ is locally transverse reducible if it is locally transverse reducible at any point; and it is locally transverse irreducible if it is not locally transversely reducible.
\end{defi}

From now on, we will simply abbreviate  ``locally transverse reducible (irreducible)" by ``\emph{locally reducible (irreducible)}" for a Sasaki manifold $(M, \xi, \eta, g)$.\\

Denote $\cD=\text{Ker}(\eta)$ to be the contact subbundle. The  \emph{transverse Levi-Civita connection} is defined as, for $Y\in\Gamma(\cD)$,
\[
\nabla^T_XY=\left\{\begin{array}{ll} &(\nabla_XY)^p,\; \mbox{if} \;X\in \cD, \\
&[\xi, Y]^p, \; \mbox{if}\;X=\xi,
\end{array}\right.
\]
where $X^p$ means the projection of $X$ on $\cD$.

\begin{defi}A vector field $Y\in \cD$ is \emph{transversely parallel} if $\nabla^T _XY=0$ for any $X$.  A subbundle $\cD_1$ of $\cD$ is  \emph{invariant} if  for any $Y\in \cD_1$, $\nabla^T_XY\in \cD_1$. The contact subbundle $\cD$ is said to be reducible (with respect to $g^T$) if there are invariant subbunldes $\cD_1, \cD_2$ with an orthogonal decomposition
 $\cD=\cD_1\oplus \cD_2$. We  call $g^T$  (irreducible) reducible if $\cD$ is (not) reducible with respect to $g^T$, and correspondingly we call $(M, \xi, \eta, g)$ \emph{transverse (irreducible) reducible}. Again we will omit the word ``transverse" from now on. 
 \end{defi}
Note that for any invariant subbundle $\cD_1$ of $\cD$, we can also define its reducibility in the above sense. Suppose $\cD$ is reducible, namely $\cD=\cD_1\oplus\cdots \cD_k$.  We say the splitting of $\cD$ is \emph{maximal} if $\cD_i$ is irreducible for each $1\leq i\leq k$.

Clearly reducibility implies locally reducibility, but the converse is in general not true. This is related with the relation between local and global holonomy. We shall return to this in Section 3.

Suppose $(M, \xi, \eta, g)$ is a compact Sasaki manifold which is transverse reducible and 
 suppose there is a maximum splitting  $\cD=\cD_1\oplus \cdots \oplus \cD_r$ with $r\geq 2$ such that $\omega^T=\omega_1\oplus \cdots \oplus \omega_r$. The main result in this section is the following

\begin{thm}\label{thm-4-1}
If each component $\omega_i$ either has positive Ricci curvature (restricted on $\cD_i$) or is flat, then $(M, \xi, \eta, g)$ is quasi-regular.
\end{thm}

First we have, 

 \begin{lem}\label{lem-4-1}
 For each $i$, the sub-bundle $\cD_i\oplus\langle\xi\rangle$ gives rise to a foliation of $M$ whose leaves are all totally geodesic.
 \end{lem}

\begin{proof}One can check this by a straightforward calculation, using the fact that $\nabla _{\widetilde X}\xi=\widetilde X$ for $\widetilde X\in\cD$, and $\nabla_{\widetilde X} \widetilde Y=\widetilde{\nabla_X Y}-\langle JX, Y\rangle \xi.$
\end{proof}
Given this, we obtain

\begin{lem} \label{lem-4-2}
If for some $i\neq j$ the transverse metric $\omega_i$ and $\omega_j$ have positive transverse Ricci curvature, then $M$ is quasi-regular.
\end{lem}

 \begin{proof}
Assume $\cD_1$ and $\cD_2$ have positive transverse Ricci curvature. By a $D$-homothetic transformation as in \cite{BGN} (proof of Theorem A), we can assume that $\cD_1$ and $\cD_2$ have positive Ricci curvature. 
 By the previous lemma any leaf of the foliation generated by $\cD_1\oplus\langle\xi\rangle$ is totally geodesic, and so the induced metric is complete with positive Ricci curvature.
 By Meyer's theorem all the leaves must be compact. Similarly this also holds for leaves of the foliation generated by $\cD_2\oplus\langle\xi\rangle$. A Reeb orbit is a transversal intersection of these two types of leaves, thus  is a compact one dimensional submanifold. This proves the lemma.
 \end{proof}

Lemma \ref{lem-4-2} shows that in general no splitting can occur for irregular manifolds with positive transverse Ricci curvature. In particular, there is no join construction for irregular Sasaki-Einstein metrics, i.e. an irregular Sasaki-Einstein manifold is locally irreducible. 

Before we prove Theorem \ref{thm-4-1}, we need  to recall some facts on transversely flat Sasaki manifolds. Let $N$ be a (possibly noncompact) Sasaki manifold with vanishing transverse curvature of dimension $2k+1$.
 By Tanno \cite{Tanno} we know the universal covering $\widetilde{N}$ is,  up to $\cD$-homothetic transformation, isomorphic to  the Euclidean space $E^{2k+1}$ with the standard Sasaki structure. In the coordinate $(x^1, \cdots, x^k, y^1, \cdots, y^k, z)$ we have
$$\eta=dz-\sum_{i=1}^ky^idx^i, \ \  \xi=\frac{\partial}{\partial z}$$ and $$g=\sum_{i=1}^k(dx^i\otimes dx^i+dy^i\otimes dy^i)+\eta\otimes \eta.$$
Denote by $G$ the automorphism group of $E^{2k+1}$ and $\g$ the Lie algebra of $G$. Write an element of $\g$ as $$X=\sum_{i=1}^k (A^i \frac{\p}{\p x^i}+B^i\frac{\p}{\p y^i})+C\frac{\p}{\p z}, $$
 then one can directly work out the condition for $X$ to preserve $g$ and $\eta$(c.f. \cite{Namba}). Any $X$ in $\g$ is generated by a function $f$ on $\C^k=\R^{2k}$ by
$$X_f=\sum_{i=1}^k (f_{y^i}\frac{\p}{\p x^i }-f_{x^i}\frac{\p}{\p y^i })+(\sum_{i=1}^ky^i\frac{\p f}{\p y^i }-f)\frac{\p}{\p z }, $$
where $f$ has the form
 $$f(x, y, z)=c_0+\sum_{i=1}^k(a_i x^i+b^i y^i)+\sum_{i,j=1}^k f_{ij}(x^ix^j+y^iy^j)+h_{ij}(x^iy^j-x^jy^i)$$ with $f_{ij}=f_{ji}$ and $h_{ij}=h_{ji}$.
  Descending to the quotient $\C^k$, constant function gives rise to zero, linear terms give rise to translations, and the  above quadratic terms give rise to unitary transformations of $\C^k$. Let $\cF_\xi$ be the Reeb foliation on $N$. Since the Sasaki manifold  $N$ is a quotient of $E^{2n+1}$, the quotient $N/\cF_\xi$ is a quotient of $\C^n$, and hence is Hausdorff. We have

\begin{lem} \label{lem-4-3}
If $N/\cF_\xi$ is compact, then $N$ is also compact, i.e. all Reeb orbits are closed.
\end{lem}

\begin{proof} If $N/\cF_\xi$ is compact, then it is a complex torus $\C^k/\Gamma$, where $\Gamma$ is a lattice. Denote the generator of $\Gamma$ by $U_i=(A_i, B_i)\in \C^k=\R^k\oplus \R^k$ for $i=1, \cdots, 2k$. It is easy to see that the lifted action of $U_i$ on $E^{2k+1}$ is given by $U_i.(X, Y, z)=(X+A_i, Y+B_i, z-B_i X^T)$, where $X, Y\in \R^k,$ and $z\in \R$.    Let $\tilde{\Gamma}$ be the subgroup of $G$ generated by $U_i$'s. Denote by $\Omega$ the standard symplectic form  on $\R^{2k}$:
$$\Omega((X_1,Y_1), (X_2, Y_2))=X_1Y_2^T-X_2Y_1^T.$$ Since $U_1$, $\cdots$, $U_{2k}$ is a real basis of $\R^{2k}$,  there is some $U_i$, say $U_2$ such that
$\Omega(U_1, U_2)\neq 0$.
It is easy to check
$$U_2. U_1. (X, Y, z)=(X+A_1+A_2, Y+B_1+B_2, z-(B_1+B_2)X^T-A_1 B_2^T, $$
therefore  we see the element $(X, Y, z)\mapsto (X, Y, z-\Omega(U_1, U_2))$ is in $\tilde{\Gamma}$. Then it is clear that $N$ is compact.
\end{proof}

This lemma implies that there is no non-compact Sasaki manifold whose transverse geometry is a compact flat torus. If such manifold existed, it would give rise to examples which contradicts Theorem \ref{thm-4-1}.  For if $N$ were such a manifold,  one could take an irregular weighted Sasaki sphere with positive transverse bisectional curvature, say $S$, and take the product $M=N\times_{\R} S$. The  $\R$ action on both $N$ and $S$ would be a proper action hence would induce a proper action on $N\times S$. The quotient space $M$ would be Hausdorff, irregular, with non-negative transverse bisectional curvature, and  reducible. \\

Now we are ready to prove Theorem \ref{thm-4-1}.  It follows essentially the same consideration in the proof of de Rham decomposition theorem, see \cite{KN}. 
\begin{proof}
By Lemma \ref{lem-4-2},  we only need to consider the case that there is exactly one factor, say $\cD_1$, with positive transverse Ricci curvature, and all others are transversely flat. Denote by $\cD'$ the direct sum of all flat factors. As in the proof of Lemma \ref{lem-4-2} all leaves of $\cD_1\oplus\langle\xi\rangle$ are compact, and the leaves of the foliation $\cD'\oplus\langle\xi\rangle$ are transversely flat. Choose one leaf of $\cD_1\oplus\langle\xi\rangle$, called $S$, with the induced Sasaki structure. Let $\pi: \widetilde{S}\rightarrow S$ be the universal covering map. $\widetilde S$  is compact again by Meyer's theorem.  Any leaf of $\cD'\oplus\langle\xi\rangle$ is transversely flat, so  by the above discussion is covered by $E^{2k+1}$.  Now we define a map $P: \widetilde S\times E^{2k+1}\rightarrow M$ as follows. Fix a point $x_0$ in $\widetilde{S}$, we have a splitting $\cD=\cD_1\oplus \cD'$ at $\pi(x_0)$. Now we choose an identification of $(\cD'\oplus \langle\xi\rangle|_{\pi(x_0)}, 0)$ with $(E^{2k+1}, 0)$ which preserves the Sasaki structure. This identification would propagate over the whole $\widetilde{S}$ by simply-connectivity and the rigidity of $E^{2k+1}$.
Then for any $(s, e)\in \widetilde S \times E^{2k+1}$ we can define $P(s,e)$ to be the exponential map at $\pi(s)$ along the direction of $e$.   It is straightforward to check that $P$ is well-defined, both open and closed, and invariant under the natural action of $\R$. Thus it descends to a local isometry (indeed a local isomorphism of Sasaki structure) $\hat{P}$ from $\widetilde{S}\times_{\R} E^{2k+1}$ to $M$, which preserves the Sasaki structure.  Thus $\hat{P}$ is a universal covering map, and $M$ is a quotient of $\widetilde{S}\times_{\R} E^{2k+1}$ by a discrete subgroup of isometries.  By compactness, the only discrete subgroup of isomorphisms of the Sasaki structure on $\widetilde{S}$ must be a finite group.

We  claim that any isomorphism of $\widetilde S \times_{\R} E^{2k+1}$  lifts to an isomorphism of $\widetilde S\times E^{2k+1}$, which is simply the product of an element in $Aut(\widetilde S)$ and an element in $Aut(E^{2k+1})$. To see this, we notice that $\widetilde S \times_{\R} E^{2k+1}$ is simply a $\widetilde S$ bundle over $\C^k$. Any isomorphism $\phi$ must preserve the distribution of positive curvature, so must map a fiber to a fiber, and descends to an isomorphism $\phi'$ of $\C^k$. On the other hand, any isomorphism $\phi'$ of $\C^k$ lifts to an isomorphism $\phi''$ of $E^{2k+1}$ by the previous discussion, and so can be viewed as an isomorphism of $\widetilde S\times E^{2k+1}$. After composing with $(\phi'')^{-1}$ we may assume $\phi$ maps every fiber to itself. On each fiber $\phi$ induces an isomorphism of $\widetilde S$. On the other hand, any isomorphism of $\widetilde S$ is automatically an isomorphism  of $\widetilde S\times_{\R} E^{2k+1}$. So after composing with such an isomorphism we may assume $\phi$ in addition is identity on one fiber $\widetilde S$. On this fiber we choose an arbitrary point $p$, then $\phi(p)=p$ and $d\phi$ is identity on the tangent space at $p$.  It follows that  $\phi$ is identity and hence we complete the proof of the claim.

Now $M$ is  a quotient of $\widetilde S\times E^{2k+1}$ by a commutative action of a discrete subgroup $G$ of isomorphisms  and $\R$.  The induced action of $G$ on $E^{2k+1}$   would then have a compact quotient. By the proof of Lemma \ref{lem-4-3}, $G$ must contain the element in $Aut(E^{2k+1})$ which is $z\mapsto z+c$ for some $c\neq 0$. Now since $M$ is Hausdorff, the quotient of $\widetilde{S}$ by the $e^{cl\xi}$ (for some nonzero $l\in \Z$) is also Hausdorff, which is equivalent to $\widetilde S$ being quasi-regular.  Thus $M$ is quasi-regular. This proves Theorem \ref{thm-4-1}.
\end{proof}

Now we assume $(M, \xi, \eta, g)$ has nonnegative transverse bisectional curvature. First we have the following, whose proof is identical to that in \cite{HSW} and so we omit:

\begin{lem}
Let $(M, \xi, \eta, g)$ be a compact Sasaki manifold with non-negative transverse bisectional curvature, then any real basic harmonic $(1,1)$ form on $M$ is transversely parallel, i.e.
if $\Delta_B\alpha=0$, then $\nabla^T\alpha=0$.
\end{lem}

Now for any real basic harmonic $(1,1)$ form $\alpha$ not proportional to $\omega^T$, the Hermitian linear transform determined by $\alpha$ with respect to $\omega^T$ will split the transverse geometry into a direct sum of orthogonal components, each of which is parallel in the transverse geometry.  It then follows that,

 \begin{lem}\label{lem-2-5}
 Let $(M, \xi, \eta, g)$ be a compact Sasaki manifold with nonnegative transverse bisectional curvature. Then $b^{1, 1}_B>1$ if and only if  $(M, \xi, \eta, g)$ is transverse reducible.
 \end{lem}
\begin{proof}
By the discussion above, if $b^{1, 1}_B>1$, then $(M, \xi, \eta, g)$ is transverse reducible. On the other hand,
suppose $(M, \xi, \eta, g)$ is transverse reducible, namely $\cD=\cD_1\oplus \cD_2$, and $\cD_i$ is invariant with respect to $g^T$. Let $\omega_i=\omega^T|_{\cD_i}$, then $\omega^T=\omega_1\oplus  \omega_2$. It is easy to see that $[\omega_i]_B\neq 0$; moreover $[\omega_1]_B$ and $[\omega_2]_B$  are linearly independent in $H^{1, 1}_B(M, \cF_\xi).$ Hence $b^{1, 1}_B\geq 2$.
\end{proof}

Now suppose $(M, \xi, \eta, g)$ be a compact Sasaki manifold with nonnegative transverse bisectional curvature with $b^{1, 1}_B>1$.  Then there exists a maximal splitting of $\cD=\cD_1\oplus \cdots \oplus \cD_k$ such that $\cD_i$ is invariant with respect to $g^T$. Let $\omega_i=\omega^T|_{\cD_i}$, then $\omega^T=\oplus_i \omega_i$. Note that the transverse Ricci form $\rho^T$ has a natural splitting such that $\rho^T=\oplus \rho_i$. For each $i$, $\rho_i$ it is a well-defined closed basic $(1, 1)$ form, hence $[\rho_i]_B$ is well-defined. Clearly $c_1^B=[\rho^T]_B=\oplus_i [\rho_i]_B$.
 \begin{lem}\label{L-maximalsplitting} Let $(M, \xi, \eta, g)$ be a compact Sasaki manifold with nonnegative transverse bisectional curvature with $b^{1, 1}_B>1$. Then there are nonnegative constants $c_i$, $1\leq i\leq k$ such that
 \[
 c^B_1=\sum_i c_i[\omega_i]_B.
 \]
 \end{lem}
\begin{proof}
We only need to prove $[\rho_i]_B=c_i [\omega_i]_B$ for each $i$. Note that $\rho_i \perp \omega_j$ for $j\neq i$, it then follows the harmonic part $H\rho_i \perp \omega_j$ by transverse Hodge theory. If $H\rho_i$ is not proportional to $\omega_i$, then the Hermitian linear transformation determined by $H\rho_i$ and $\omega_i$ on $D_i$ will split $\cD_i$ into a direct sum of orthogonal components, contradiction since $\cD_i$ is irreducible. The non-negativity of $c_i$ follows directly from the non-negativity of the transverse bisectional curvature.
\end{proof}

The following lemma follows from the transverse Calabi-Yau theorem. For convenience we include a sketch of proof. 
 \begin{lem} \label{Calabi-Yau} There exists a Sasaki structure $(\xi, \tilde \eta, \tilde g)$ on $M$, which is a transverse K\"ahler deformation of $(
\xi, \eta, g)$, such that
\[
\widetilde{ Ric^T}=\sum_{i=1}^r c_i \omega_i.
\]
Moreover, the contact subbundle $\tilde \cD=\text{Ker}(\tilde \eta)$ admits a maximal splitting $\tilde \cD=\tilde\cD_1\oplus \cdots \oplus \tilde \cD_r$ such that \[\text{span}\{\xi, \tilde \cD_i\}=\text{span}\{\xi, \cD_i\}\] for any $1\leq i\leq r$.
\end{lem}
\begin{proof}The first part simply follows from the transverse Calabi-Yau theorem in \cite{Eka90}. It asserts that for any real basic $(1, 1)$ form $\theta$ in $c_1^B$, there exists a unique Sasaki structure $(\xi, \tilde \eta, \tilde g)$ with the same transverse complex structure of $(\xi, \eta, g)$, such that 
its Ricci form $\tilde \rho$ given by $\theta$. Hence we can take $\theta=\sum c_i \omega_i$. Moreover since the Ricci form $\tilde \rho$ splits, then the transverse metric $\tilde g^T$  splits correspondingly. To see this fix an $i$ between $1$ and $r$, denote by $f_i$ the normalized potential such that
\[
\rho_i-c_i\omega_i=\i \p\bar \p f_i, 1\leq i\leq r. 
\]
 By \cite{Eka90}, we can solve the equation for
\[
\eta\wedge (d\eta+\i\p\bar\p \phi_i)^n=\eta\wedge (d\eta)^n \exp(f_i). 
\]
Now by the maximum principle, we know that for any \[Y\in \text{span}\{\xi, \cD_1, \cdots, \hat \cD_i, \cdots, \cD_r\},\]  then $Y\phi_i=0$. Hence we know that for each $i$, there exists $\phi_i$ such that $\tilde \eta=\eta+d^c\phi$, where $\phi=\sum_i\phi_i$; and  $d\tilde \eta=\oplus_i (\omega_i+\i\p\bar \p \phi_i)$. In particular, $\tilde g^T$ splits correspondingly, and the second part of the proposition follows.

\end{proof}

We need to recall the following notion.
 \begin{defi}
 A Riemannian metric $(M, g)$ is \emph{locally symmetric} if the curvature satisfies $\nabla Rm=0$. A Sasaki structure $(M, \xi, \eta, g)$ is called \emph{transversely symmetric} if  any induced transverse K\"ahler metric $\{g^T_\alpha, V_\alpha\}$  is locally symmetric (locally Hermitian symmetric).
 \end{defi}

By  \cite{Taka}, in terms of our terminology, we have

\begin{thm}[Takahashi] \label{T-symmetric} A compact simply connected transverse symmetric Sasaki manifold is homogeneous. Hence it is  a principal $S^1$ bundle over a Hermitian symmetric space. \end{thm}

\section{Proof of Theorem \ref{T-general}}

We shall first prove the special case Theorem \ref{T-special}.  And  Theorem \ref{T-general} follows easily from Theorem \ref{thm-4-1} and this special case.

First  we assume $(M, \xi, \eta, g)$  is a compact simply-connected Sasaki manifold with nonnegative transverse bisectional curvature and $b^B_2=1$. So there is a $\lambda\geq0$ such that 
\[
c_1^B=\l [\omega^T]_B.
\]
When $\l=0$, then 
\[
\rho^T+\i \p\bar \p f=0
\]
for some basic function $f$. Since $Ric^T\geq 0$, $f$ must be  a constant, hence $\rho^T=0$. It follows further that $(M, \xi, \eta, g)$ is transverse flat. The only simply-connected transverse flat Sasaki manifold is $\R^{2k+1}$ with its standard Sasaki metric(c.f. \cite{Tanno}), which contradicts the compactness assumption.
Hence by a scaling we can assume $\l=2n+2$ and we consider the unnormalized Sasaki-Ricci flow
\begin{equation}\label{E-ricciflow}
\frac{\p g^T}{\p t}=-Ric^T
\end{equation}
with initial metric $g$. 
The equation \eqref{E-ricciflow} has short time existence by \cite{SWZ}. 
By maximum principle as in \cite{Mok} $g(t)$ has nonnegative transverse bisectional curvature for $t>0$.
Since $c_1^B>0$ implies the transverse Ricci curvature of $g$ is positive at least at one point,  it then follows that  $g(t)$ has positive transverse Ricci curvature and positive transverse holomorphic sectional curvature for $t\in (0, \delta]$, where $\delta>0$ is a small number. 

Now we suppose  the transverse K\"ahler metric $g^T$ of $(M, \xi, \eta, g)$ is not locally transverse symmetric. By continuity and by making $\delta$ smaller we can assume $g^T(t)$ is not locally transverse symmetric for $t\in (0, \delta)$. In the following we want to prove

\begin{prop}\label{P-positive}Under the above assumptions,  if $g(t)$ is not locally transverse symmetric, then $g(t)$ has positive transverse bisectional curvature for $t\in (0, \delta)$.
\end{prop}

Note that along the Sasaki-Ricci flow, the contact 1-form $\eta$, hence the contact subbundle $\cD$ varies along the flow, while the quotient bundle $\nu=\nu(\cF_\xi)$ remains unchanged. These two bundles are identified by $f$, as in Section 2. Now we use a technique as in Hamilton \cite{Hamilton88}. We  fix an Hermitian metric $h$ on $\nu^{1,0}$. Consider the bundle isomorphisms  $f_t: \nu^{1, 0}\rightarrow \cD^{1, 0}$ such that
\[
h=f_t^{*} g^T(t).
\]
Given a local trivialization of $\nu$ and $\cD$, the isomorphism $f_t$ satisfies the evolution equation
\[
\frac{\p f_\alpha^i}{\p t}=\frac{1}{2}g^{i\bar j}_T R^T_{k\bar j} f^k_{\alpha}.
\]
By pulling back the transverse connection $\nabla^T_t$ of $g^T(t)$ through $f_t$ ($t\in (0, \delta)$), we can then define the covariant derivatives $D_t$ for sections of $\nu$ and any tensor bundle generated by $\nu$. We can also pull back  the curvature tensor $R^T_{i\bar j k\bar l}$ to a tensor on $\nu\otimes \C$ by \[R^\nu_{\alpha\bar \beta \gamma\bar \delta}=R^T_{i\bar jk\bar l}f^i_\alpha \bar{f^j_\beta} f^k_\gamma \bar{f^l_\delta}.\] Note that in local coordinates, the transverse K\"ahler metric $\{g^T_\alpha, V_\alpha\}$ satisfies the K\"ahler-Ricci flow
\begin{equation}\label{E-3-1}
\frac{\p g^T_\alpha}{\p t}=-Ric^T_\alpha.
\end{equation}
It  then follows from the standard calculation that the curvature tensor $R^\nu_{\alpha\bar \beta\gamma \bar \delta}$ satisfies, on the orthonormal frame $\{e_\alpha\}$ of $\nu^{1, 0}$,
\begin{equation}\label{E-3-2}
\frac{\p R^\nu_{\alpha\bar \alpha \beta\bar \beta }}{\p t}=\Delta_t R^\nu_{\alpha\bar \alpha \beta\bar \beta }+\sum_{\zeta, \mu}\left({R^\nu_{\alpha\bar \alpha \mu \bar \zeta}} R^\nu_{\zeta \bar \mu\beta\bar \beta}-|R^\nu_{\alpha\bar \zeta \beta \bar \mu}|^2+|R^\nu_{\alpha\bar \beta \mu \bar \zeta}|^2\right),
\end{equation}
where $\Delta_t$ is the Laplacian induced by the connection $D_t$ on sections of $((\nu^{1, 0})^{*})^{\otimes 4}$.

As in \cite{Gu} Section 2, one obtains that for any point $p\in M$ and unit vectors $e_\alpha, e_\beta\in \nu^{1, 0}_p$, it holds
\begin{equation}\label{E-gu}
\sum_{\mu, \zeta}\left( R^\nu_{\alpha\bar\alpha \mu\bar \zeta} R^\nu_{\zeta\bar \mu \beta\bar\beta}-|R^\nu_{\alpha\bar \mu \beta \bar \zeta}|^2\right)\geq  \min\{0, A\},\end{equation}
where for any $\zeta, \mu\in \nu^{1, 0}_p$ for $p\in M$, there is a universal positive constant $c$ such that
\[
A\geq c\cdot \inf _{|X|=|Y|=1}\frac{d^2}{ds^2}|_{s=0}R^\nu(e_\alpha+s X, \overline{e_\alpha+s X}, e_\beta+s Y, \overline{e_\beta+sY}).
\]

Now we use  the language as \cite{BS2}. Define $P$ to be the orthonormal  frame bundle of $\nu^{1, 0}$ with structure group $U(n)$, which acts on P on the right. For each $t$, the connection $D_t$ on sections of $\nu$ induces a connection on $P$. For each point $\underline e=(p,\{e_1, e_2, \cdots, e_n\})\in P$, let $\mathbb{H}_{\underline e}$ be the horizontal distribution and $\mathbb{V}_{\underline e}$ be the vertical subspace of $T_{\underline e} P$; hence $T_{\underline e}P=\mathbb{H}_{\underline e}\oplus \mathbb{V}_{\underline e}$. Next we define a collection of smooth horizontal vector fields $\tilde X_1$, $\cdots$, $\tilde X_n$, $\tilde Y_1$, $\cdots$, $\tilde Y_n$ and $\tilde \xi$ on $P$.  Note that $\{f (e_1), \cdots, f (e_n)\}$ span $\cD^{1, 0}$.
Denote \[X_i=\frac{1}{2} (f(e_i)+f(\bar e_i)), \; \mbox{and}\; Y_i=\Phi X_i=\frac{-\i}{2}(f(e_i-\bar e_i)), i=1, \cdots,  n.\]
It then follows that $TM$ is spanned by $\{\xi, X_1, \cdots, X_n, Y_1, \cdots, Y_n\}$. We then define $\tilde \xi$ to be the horizontal lift of $\xi$, and $\tilde X_i$, $\tilde Y_i$ to be the horizontal lift of $X_i$, $Y_i$ respectively.\\

We define a nonnegative function $u: P\times (0, \delta)\rightarrow R$ by
\[
u: (\underline e=\{e_1, \cdots, e_n\}, t)\rightarrow R^{\nu}_{g(t)}(e_1, \bar e_1, e_2, \bar e_2).
\]
By \eqref{E-3-2} and \eqref{E-gu}, we then get
\begin{equation}
\frac{\p u}{\p t} \geq Lu-c\min\{0, \inf_{|\zeta|=1, \zeta\in \mathbb{V}_{\underline e}} D^2u(\underline e, t)(\zeta, \zeta)\},
\end{equation}
where \[Lu=\sum_i\tilde X_i (\tilde X_i u)+\sum_i \tilde Y_i (\tilde Y_i u)+\tilde \xi (\tilde \xi u)\] denotes the horizontal Laplacian on $P$ and we have used the fact that all metrics $g(t)$ are $\xi$-invariant.

Denote  $F=\{(\underline e, t): u(\underline e, t)=0\}\subset P\times (0, \delta)$.
Invoking Proposition 5 in  \cite{BS2},  we have
\begin{prop}\label{P-parallel}
For any fixed $t\in (0, \delta)$, if $\tilde \gamma: [0, 1]\rightarrow P$ is a smooth horizontal curve satisfying $(\tilde \gamma(0), t)\in F$, then $(\tilde \gamma(1), t)\in F$.
\end{prop}

Suppose for some $t\in(0, \delta)$ that $g(t)$ does not have positive transverse bisectional curvature. We can then assume that for some $(\underline e, t) \in P\times (0, \delta)$,  $R^\nu(e_1, \bar e_1, e_2, \bar e_2)=0$.

Next we need the following holonomy theorem,
\begin{prop}\label{P-holonomy}For any fixed $t\in (0, \delta)$, we have
 \[\textsf{Hol}(P)=\textsf{Hol}(\nu, D_t)=\textsf{Hol}(\cD_t, \nabla^T_t)=\textsf{Hol}(g^T_{\alpha,t}, V_\alpha)=U(n),\]
where $\textsf{Hol}(g^T_{\alpha,t}, V_\alpha)$ denotes the holonomy of transverse K\"ahler metric $g^T_{\alpha,t}$ in any (simply connected) local coordinate $V_\alpha$.
\end{prop}

\begin{proof}
First observe that $g(t)$ and $g^T(t)$ are analytic in normal coordinate for $t\in (0, \delta)$. Bando \cite{Bando1} proved the analyticity of the metric $g(t)$ along the Ricci flow. His proof can be modified directly to Sasaki setting. It then follows that $\nabla^T$ is analytic.
Now let $U_\alpha\subset M$ be an open set which is diffeomorphic to $\R^{2n+1}$ and $\pi_\alpha: U_\alpha\rightarrow V_\alpha$ be the submersion correspond to Reeb foliation, as introduced in Section 2. We can restrict our consideration on $U_\alpha$ to talk about local holonomy, then by definition we have
\[\textsf{Hol}(P_{U_\alpha})=\textsf{Hol}(\nu_{U_\alpha}, D_t)=\textsf{Hol}(\cD_{U_\alpha}, \nabla^T).
\]
Since $\nabla^T$ is induced by $g^T$ and $\nabla^T_{\xi} \cD=0$, we have
\[
\textsf{Hol}(\cD_{U_\alpha}, \nabla^T)=\textsf{Hol}(g^T_\alpha, V_\alpha).
\]
By the analyticity of the connection $\nabla^T$, it then follows that the local honolomy coincides with the global holonomy (\cite{KN}, Chapter II, Section 10).
In particular,
\begin{equation}\label{E-3-7}\textsf{Hol}(\cD, \nabla^T)=\textsf{Hol}(\cD_{U_\alpha}, \nabla^T)=\textsf{Hol}(g^T_\alpha, V_\alpha).
\end{equation}

\end{proof}

Lemma \ref{lem-2-5} then implies $\cD$ is transverse irreducible. It then follows \eqref{E-3-7} that $g^T_\alpha$ is irreducible. Note that we also assume that $g^T$ is not locally symmetric, hence $g^T_\alpha$ is not locally symmetric.
Now by Berger's theorem, we know that
$\textsf{Hol}(g^T_\alpha, V_\alpha)=U(n), SU(n), \; \mbox{or}\; Sp(n/2).$ Clearly the later two cases imply that $g^T$ is transverse Ricci flat and can not occur.

By Proposition \ref{P-parallel} and \ref{P-holonomy}, we know that $\{\tilde {\underline e}, t\}\in F$ if $\tilde{\underline e}=\underline e A$ for $A\in U(n)$. Let
\[
\begin{pmatrix}
\tilde {e_1}&
\tilde e_2
\end{pmatrix}=\begin{pmatrix}
e_1 &
e_2
\end{pmatrix} \cdot \begin{pmatrix}
 \cos \theta & -\sin\theta\\
\sin\theta & \cos\theta
\end{pmatrix}
\]

Then we get $R^{\nu}(\tilde e_1, \overline {\tilde e_1}, \tilde e_2, \overline {\tilde e_2})=0$.
However, by \eqref{E-3-2} and the first variation of $R^{\nu}_{\alpha\bar \alpha \beta \bar\beta}$ at $\nu_p^{1, 0}$ (see \cite{Gu}), if $R^\nu_{1\bar 1 2\bar 2}=0$, we have
\begin{equation}\label{E-null}
\left\{
\begin{array}{lll}
\sum_{\mu, \zeta}\left(R^\nu_{1\bar 1 \mu\bar \zeta} R^\nu_{\zeta\bar \mu 2\bar 2}-|R^\nu_{1\bar \mu 2 \bar \zeta}|^2\right)=0\\
R^\nu_{1\bar 2 \mu\bar \zeta}=0, \;\forall \mu, \zeta.\\
R^{\nu}_{1\bar 1 2\bar \mu}=R^\nu_{2\bar2 1\bar \mu}=0, \;\forall \mu.
\end{array}\right.
\end{equation}
On the other hand, a direct computation (using \eqref{E-null}) implies
\[
R^\nu(\tilde e_1, \overline{\tilde e_1}, \tilde e_2, \overline{\tilde e_2})=\cos ^2\theta\sin^2\theta (R^\nu_{1\bar 1 1\bar 1}+R^\nu_{2\bar 2 2\bar 2}).
\]
This contradicts with the positivity of  transverse holomorphic sectional curvature of $g(t)$ for any $t\in (0, \delta)$.
It completes the proof of Proposition \ref{P-positive}. \\

Theorem \ref{T-special} is now a direct consequence. If $(M, g)$ is not locally transverse symmetric, by Proposition \ref{P-positive} and our previous results \cite{He-Sun}, $(M, g)$ is then a simple metric on a weighted Sasaki-sphere.  If $(M, g)$ is locally transverse symmetric, by Takahashi's result (Theorem \ref{T-symmetric}), then $M$ is homogenous and hence it is a $S^1$ bundle over an Hermitian symmetric space. \\

Now we finish the proof of Theorem \ref{T-general}. Suppose $(M, \xi, \eta, g)$ is a compact Sasaki manifold with nonnegative transverse bisectional curvature, with a maximum splitting  $\cD=\cD_1\oplus \cdots \oplus \cD_r$ such that $\omega^T=\omega_1\oplus \cdots \oplus \omega_r$. If $r=1$, then by Lemma \ref{lem-2-5} $b^{1, 1}_B=1$. Hence  $c_1^B=\l [\o]$ for $\l=0$ or $\l>0$. When $\l=0$, by nonnegative curvature assumption it follows that the transverse bisectional curvature has to be flat. The universal cover of $M$ is then isometric to $E^{2n+1}$ with standard Sasaki structure \cite{Tanno}. 
When $\l>0$,  then $M$ is compact and $\pi_1(M)$ is finite. Passing to universal covering it is  reduced to Theorem \ref{T-special}.  So we only need to deal with the case $r\geq 2$.  By Lemma \ref{L-maximalsplitting}, we may assume that
 \[
 c_1^B=\sum_{i=1}^r c_i [\omega_i]_B.
 \]
 By Lemma \ref{Calabi-Yau}, after a transverse K\"ahler deformation we obtain a Sasaki structure $(\xi, \tilde{\eta}, \tilde{g})$ which satisfies the assumption of Theorem \ref{thm-4-1}. It follows that $(M, \xi, \tilde \eta, \tilde g)$ is quasi-regular; hence $(M, \xi, \eta, g)$ is also quasi-regular. 
 
We assume the splitting of $\cD$ given by $\cD=\cD_0\oplus \cD_1$ and  $\cD_0$ is the maximal subbundle such that $g$ is transverse flat on $\cD_0$.  The foliations generated by $\cD_1\oplus \langle \xi\rangle$ are all isometric to a (quasi-regular) compact Sasaki manifold  with $c_1^B$ positive and nonnegative transverse bisectional curvature $(S, h)$; hence its universal covering $(\tilde S, \tilde h)$ is compact. Passing to the universal covering and following the proof in Theorem \ref{thm-4-1}, then $\widetilde M$ is isometric to  $(E^{2l+1}, g_0)\times_{\R^1} (\tilde S, \tilde h)$, a join construction. If $(\tilde S, \tilde h)$ is reducible, then it follows that similarly $(\tilde S, \tilde h)=(S_1, g_1)\times_{S^1}(S_2, g_2)$, where $(S_i, g_i)$ is a quasi-regular,  compact simply-connected Sasaki manifold with nonnegative transverse bisectional curvature. We can then keep doing this until each piece $S_i$ such that $b^{1, 1}_B(S_i)=1$. 
By Theorem \ref{T-special}, $(S_i, g_i)$ is either a (quasi-regular) weighted Sasaki sphere, or a regular compact Sasaki manifold which corresponds to a compact Hermitian symmetric space with rank $\geq 2$. In particular, we have

\begin{thm}\label{thm-4-2} Let $(M, \xi, \eta, g)$ be a quasi-regular compact Sasaki manifold with non-negative transverse bisectional curvature. Then  its universal cover $\widetilde{M}$ is isomorphic to a join of $(S_1, g_1)$, $\cdots$, $(S_k, g_k)$ with $(E^{2l+1}, g_0)$ for some $k, l\geq 0$, where $(S_i, g_i)$ is a compact simply-connected quasi-regular Sasaki manifold with nonnegative transverse bisectional curvature and with $b^2_B=1$, $(E^{2l+1}, g_0)$ is the standard Sasaki metric on $\R^{2l+1}$. \end{thm}

Theorem \ref{thm-4-2} is the Sasaki version of Theorem \ref{T-general} (2) when $M$ is assumed to be quasiregular. 
We should emphasize that  there is no canonical choice of the quotient of $S^1$ action in the above construction. This corresponds to the fact that there is no canonical  join construction for quasiregular Sasaki manifolds. For more details about join construction, see \cite{BG}, Section 7.6 for example.

Weiyong He\\
Department of Mathematics, University of Oregon, Eugene, Oregon, 97403\\
whe@uoregon.edu

Song Sun\\
Department of Mathematics, Imperial College, London, SW7 2AZ, U.K.\\
s.sun@imperial.ac.uk

\end{document}